\documentclass[12pt,a4paper,reqno]{amsart}
\usepackage{amsmath,amssymb,amsthm,amsfonts}
\usepackage[export]{adjustbox}
\usepackage{amsbsy}
\usepackage[utf8]{inputenc}
\usepackage[normalem]{ulem}
\usepackage{xspace}
\usepackage{cancel}
\usepackage[greek,english]{babel}
\usepackage{url}
\usepackage{wrapfig}
\usepackage{enumitem}
\usepackage{multicol}
\usepackage{moreenum}
\usepackage{xcolor}
\usepackage[pagewise]{lineno}

\usepackage{xcolor}
\newcounter{dummy}
\makeatletter
\newcommand\myitem[1][]{\item[#1]\refstepcounter{dummy}\def\@currentlabel{#1}}
\makeatother
\usepackage{subcaption}
\usepackage{tikz}
\usepackage{float}
\usepackage[bottom=2.5cm,top=2.5cm, left=2.5cm, right=2.5cm]{geometry}

\usepackage{multirow}
\usepackage{array,multirow}

\definecolor{LinkColor}{rgb}{0,0,0} 
\usepackage[colorlinks=true,linkcolor=LinkColor,citecolor=LinkColor,urlcolor= LinkColor, naturalnames, hyperindex, pdfstartview=FitH, bookmarksnumbered, plainpages]{hyperref}
\usepackage{cleveref}

\makeatletter
\newcommand{\slunlhd}{%
	\mathrel{\mathpalette\sl@unlhd\relax}%
}

\newcommand{\sl@unlhd}[2]{%
	\sbox\z@{$#1\lhd$}%
	\sbox\tw@{$#1\leqslant$}%
	\dimen@=\ht\tw@
	\advance\dimen@-\ht\z@
	\ifx#1\displaystyle
	\advance\dimen@ .2pt
	\else
	\ifx#1\textstyle
	\advance\dimen@ .2pt
	\fi
	\fi
	\ooalign{\raisebox{\dimen@}{$\m@th#1\lhd$}\cr$\m@th#1\leqslant$\cr}%
}
\makeatother

\newtheorem{theorem}{Theorem}[section]
\newtheorem{corollary}[theorem]{Corollary}

\newtheorem{proposition}[theorem]{Proposition}

\theoremstyle{definition}
\newtheorem{definition}[theorem]{Definition}
\newtheorem{remark}[theorem]{Remark}

\newtheorem{example}[theorem]{Example}

\newcommand{\ecut}{\textsf{ecut}\xspace}
\newcommand{\cut}{\textsf{cut}\xspace}

\newcommand{\Irr}{\operatorname{Irr}}

\newcommand{\U}{\textup{U}}

\newcommand{\ZZ}{\mathbb{Z}}

\newcommand{\Q}{{\mathbb Q}}
\newcommand{\R}{{\mathbb R}}

\definecolor{wildstrawberry}{rgb}{1.0, 0.26, 0.64}

\newcommand{\GEN}[1]{\langle #1 \rangle}

\title{Extended \texttt{cut} groups}

\author[\`{A}. Garc\'{i}a-Bl\`{a}zquez]{\`{A}ngel Garc\'{i}a-Bl\`{a}zquez}
\address{\`{A}ngel Garc\'{i}a-Bl\`{a}zquez}
\email{\href{mailto:angel.garcia11@um.es}{angel.garcia11@um.es}}

\author[G. Kaur]{Gurleen Kaur}
\address{(Gurleen Kaur) Indian Institute of Technology Ropar, Punjab 140001, India}
\email{\href{mailto:gurleenkaur992gk@gmail.com}{gurleenkaur992gk@gmail.com }}

\author[S. Maheshwary]{Sugandha Maheshwary}
\address{(Sugandha Maheshwary) Indian Institute of Technology Roorkee, Haridwar, PIN 247667, Uttarakhand, India.}
\email{\href{mailto:msugandha@ma.iitr.ac.in}{msugandha@ma.iitr.ac.in}}

\thanks{The second author acknowledges the research support of the Department of Science and Technology (INSPIRE Faculty No. DST/INSPIRE/04/2023/001200), Govt. of India.The third author gratefully acknowledges the support by Science \& Engineering Research Board (SERB),  Department of Science and Technology (DST), India (SRG/2023/000180).}


\begin{document}
	
	\maketitle

	\begin{abstract} A finite group $G$ is said to be a \cut group if  
		all \textbf{c}entral \textbf{u}nits in the integral group ring $\mathbb{Z}G$ are \textbf{t}rivial.
		In this article, we extend the notion of \cut groups, by introducing extended \cut groups. We study the properties of extended \cut groups analogous to those known for \cut groups and also characterise some substantial classes of groups having the property of being extended \cut. A complete classification of extended \cut split metacyclic groups has been presented.
		
	\end{abstract}

	\section{Introduction}
	Let $G$ be a finite group and let $\mathcal{Z}(\U(\mathbb{Z}G))$ denote the centre of the unit group $\U(\mathbb{Z}G)$ of the integral group ring $\mathbb{Z}G$. The work on understanding the $\U(\mathbb{Z}G)$ and $\mathcal{Z}(\U(\mathbb{Z}G))$ is immense. Higman's work of 1940 \cite{Hig40} seemingly has been an initial masterpiece. Since then, several problems in this direction have been raised and attempted and the subject fascinates till date. 
		A glimpse on the current state of art can be found in \cite{MP18} and \cite{Jes21}.

	It is a well known fact that  $\mathcal{Z}(\U(\mathbb{Z}G))=\pm \mathcal{Z}(G)\times A_G$, where the torsion part $\pm \mathcal{Z}(G)$ contains the trivial central units of $\ZZ G$ (scalar multiples of central group elements), and the non-torsion central units of $\mathbb{Z}G$ are contained in $A_G$, which is a free abelian group of finite rank, say $\rho(G)$ (see \cite{JdR16}, Section 7.1 for details). This leads to a natural interest of determining $\rho(G)$, for a given group $G$. For an abelian group $G$, $\rho(G)$ was given by Ayoub and Ayoub \cite{AA69}. In independent works, Ritter and Sehgal \cite{RS05} and Ferraz \cite{Fer04} determined $\rho(G)$  in terms of conjugacy classes of $G$. Particularly, for split metacyclic groups, $\rho(G)$ has been computed in \cite{FS08}. For the so-called strongly monomial groups, $\rho(G)$ in terms  of special pairs of subgroups, called strong Shoda pairs of $G$,
	 was given by Jespers et al. \cite{JOdRG13}. This was further generalized to even bigger class, namely generalized strongly monomial groups by Bakshi and Kaur \cite{BK22}. It may be noted that the class of generalized strongly monomial groups is quite dense in the sense that 
	   it seems to be a challenging task to find an example of a monomial group which is not generalized strongly monomial \cite{Kau23}.  
	
	The groups $G$ for which $\rho(G)=0$ are particularly of special interest. For, $\rho(G)=0$ is equivalent to saying that all \textbf{c}entral \textbf{u}nits in the integral group ring are \textbf{t}rivial (equivalently, torsion). Such a group $G$ is said to be a \cut group, as coined by Bakshi et al \cite{BMP17}.  Interestingly, though the term \cut is coined only in recent years, such groups under different names and perspectives have been of interest for a long time. To the best of our knowledge, the study of such groups dates back to 1970’s (c.f.\cite{BBM20}) as observed in the work of Patay (a PhD student of A. Bovdi) who studied when  a finite simple group is a \cut-group and answered it for alternating groups \cite{Pat75,Pat78}. 
	The prominent appearances of questions on \cut groups and initial developments were then seen in \cite{GP86}, \cite{Bov87} and \cite{RS90}. The study of \cut groups involves a rich interplay of group theory, module theory, representation theory, algebraic number theory and K-theory. For details on properties of \cut groups, we refer Section 3 of \cite{MP18}. The class of \cut groups is of great interest, as can be observed from some recent works \cite{BMP17, Mah18,Bac18,Bac19,BMP19,Tre19,Gri20,BBM20,BCJM21,Mor22,BKMdR23}.
	
	In this article, we extend the notion of \cut groups, by introducing extended \cut groups, defined as the groups $G$ with $\rho(G)\leq 1$. We call extended \cut groups as \ecut groups for brevity. Indeed, \cut groups are \ecut groups. It may be pointed out that for any $n\in\mathbb{N}$, symmetric group $S_n$ is a \cut group and hence an \ecut group. The classification of the alternating groups $A_n$  with $\rho(A_n) \leq 1$ is known, i.e., $n$ for which $A_n$ is \cut or \ecut are known \cite{AKS08}. Also, finite simple \ecut groups have been listed by Bovdi et al. \cite{BBM20}. These works also motivated the study of groups $G$ with $\rho(G)$ at most $1$. Note that for an \ecut group $G$, the normalized group of central units of the integral group ring $\mathbb{Z}G$ is an infinite cyclic group, and one may like to write its generator (non-torsion unit). Such works appear in \cite{Ale94}, for some cyclic and alternating \ecut groups. For an \ecut group $G$, one may compute this generator using several other techniques known (\cite{JPS96,LP97,FS08,JP12,JOdRG13,JdRV14,BK19}).

	The aim of this article is to study finite \ecut groups and we focus on solvable groups. We set up notation and preliminaries in Section 2, and therein also recall that the class of \cut groups is precisely the class of inverse semi-rational groups, which is a generalization of well studied class of rational groups. In Section 3,
	we discuss the criteria for a group to be \ecut and classify finite abelian \ecut groups. This is followed by characterising finite nilpotent \ecut groups in Section $4$. Further, we study monomial \ecut groups in Section 5, where we provide ample examples of \ecut groups. In particular, we give a complete classification for split metacyclic \ecut groups. We conclude the article by discussing the prime spectrum of a solvable \ecut group, in Section 6.	

	\section{Notation and preliminaries}
	 
	 For a group $G$ and for $g\in G$, $|g| $ and $|G|$ denote the order of $g$ and $G$ respectively. If $g,h\in G$ and $g$ is conjugate (not conjugate) to $h$ in $G$, then we write $g\sim h$ ($g\nsim h$). Let $C_g$ denote the conjugacy class of an element $g$ in $G$ and let $e$ denote the exponent of $G$. The $\R$-class of $g$, denoted by $C_{g}^{\mathbb{R}}$ equals $C_g\cup C_{g^{-1}},$ and the $\Q$-class of $g$, denoted by $C_{g}^{\mathbb{Q}}$, is 
	  $\cup_{i}C_{g^{i}}$, where $i \in \{1\leq i \leq e: \gcd(i,e)=1\}$. We use $H\leq G$ ($H\unlhd G$) to state that $H$ is a subgroup (normal subgroup) of $G$.
	 
	 In the study of the \cut and \ecut groups, often some classes of groups come into picture, which are defined by their character values, or the conjugacy relations satisfied by their elements. These classes are closely related and one switches from one form to the other quite frequently. For the convenience of the reader, we first put together the relevant terms explicitly and concisely.\\
	 
	 Let $\Irr(G)$ denote the set of all irreducible complex characters of $G$. For $\chi \in \Irr(G)$, and for $x\in G$, set
	 \begin{quote}
	 	$\mathbb{Q}(\chi):=\mathbb{Q}({\chi(g)\,|\,g\in G})$;\\
	 $\mathbb{Q}(x):=\mathbb{Q}({\chi(x)\,|\,\chi\in \Irr(G)}).$
	 \end{quote}
	 
	 \begin{definition}(Real element, real valued character and real group) An element $x\in G$ is said to be real, if $x\sim x^{-1}$. Equivalently, $C_x=C_x^{\mathbb{R}}$. Also, equivalent to saying that  $\mathbb{Q}(x)\subseteq \mathbb{R}.$ A character $\chi\in \Irr(G)$ is said to be real valued if $\mathbb{Q}(\chi)\subseteq \mathbb{R}.$ A group $G$ is said to be real if every element of $G$ is real, or equivalently, if every irreducible character of $G$ is real valued. 
	 	
	 \end{definition}
	 It is well known that the number of real valued irreducible complex characters of $G$ is equal to the number of conjugacy classes of $G$, consisting of real elements (see for instance, \cite{Isa06}(6.13), p.\,96).

	 \begin{definition}(Rational element, rational valued character and rational group)
	 	An element $x\in G$ is said to be rational, if for every $j$, such that $\gcd(j,|x|)=1$, we have that $x^j\sim x$. Equivalently, $C_x=C_x^{\mathbb{Q}}$. This is also equivalent to saying that  $\mathbb{Q}(x)= \mathbb{Q}$ (\cite{Hup67}, Theorem 13.7, p.\,537). A character $\chi\in \Irr(G)$ is said to be rational valued (or simply rational) if $\mathbb{Q}(\chi)= \mathbb{Q}.$ A group $G$ is said to be rational if every element of $G$ is rational, or equivalently, if every irreducible character of $G$ is rational.

	 \end{definition}
 	Observe that rational characters of a group correspond to the rational rows in the character table, while the rational classes correspond to the rational columns.
	 In general, the number of rational irreducible characters of a group $G$ is not
	 same as the number of conjugacy classes of $G$, consisting of rational elements (see e.g. \cite{BCJM21}, Example 3.3). 
	 
	 \begin{definition}(Inverse semi-rational element and group)
	 	An element $x\in G$ is said to be inverse semi-rational, if for every $j$, such that $\gcd(j,|x|)=1$, we have that $x^j\sim x$ or $x^j\sim x^{-1}$. Equivalently, $C_x^{\mathbb{Q}}=C_x^{\mathbb{R}}$. A group $G$ is said to be inverse semi-rational if every element of $G$ is inverse semi-rational.
	 	
	 \end{definition}
	 In view of (\cite{BCJM21}, Propositions 2.1 and 3.2), the following are equivalent definitions for an inverse semi-rational group $G$.
	 
	 \begin{enumerate}
	 	\item[(i)] $G$ is a \cut group.
	 	\item[(ii)] $\mathbb{Q}(x)$ is either $\mathbb{Q}$ or imaginary quadratic for every $x\in G$.
	 	\item[(iii)] $\mathbb{Q}(\chi)$ is either $\mathbb{Q}$ or imaginary quadratic for every $\chi\in \Irr(G)$.\\
	 	
	 \end{enumerate}

	 Further, it is known that for an inverse semi-rational group $G$,  the number of rational irreducible characters of $G$ is equal to the number of conjugacy classes of $G$, consisting of rational elements (\cite{BCJM21}, Theorem 3.1).
	 
	 \begin{definition}(Semi-rational element and group)
	 	An element $x\in G$ is said to be semi-rational, if for every $j$, such that $\gcd(j,|x|)=1$, we have that $x^j\sim x$ or $x^j\sim x^{m}$, for some $m$. A group $G$ is said to be semi-rational if every element of $G$ is semi-rational.
	 	
	 \end{definition}
	 
	 \noindent Clearly, for $x\in G$, 
	 
	 \begin{quote}
	 	$x$ is rational $\implies~x$ is inverse semi-rational $\implies~x$ is semi-rational;
	 	 \end{quote}
 	 and 
	 		 \begin{quote}
	 	$x$ is inverse semi-rational as well as real $\implies~x$ is rational.\\
	 \end{quote}
	
	  It is known that if $|G|$ is odd, then $G$ is semi-rational group if and only if $G$ is inverse semi rational (\cite{CD10}, Remark 13).

	 \begin{definition}(Quadratic character, quadratic rational group)
	 	A character $\chi$ of $G$ is quadratic, if the degree of $\mathbb{Q}(\chi)$ over $\mathbb{Q}$ is $2$, and a group $G$ is called quadratic rational if every $\chi \in \Irr(G)$, $\mathbb{Q}(\chi)$ is either rational or quadratic. 
	 \end{definition}
	 
	 It is known that if $G$ is an odd order group, then $G$ is semi-rational group if and only if $G$ is quadratic rational. However, it is not true in general (\cite{Ten12}, Section 6).

		\section{\cut groups and \ecut groups} 
		
	Throughout the article, groups are always assumed to be finite, and characters are complex. 
		
		Let $G$ be a group. As mentioned in the introduction, $\rho(G)$ denotes the rank of free abelian part of $\mathcal{Z}(U(\mathbb{Z}G))$. 
		If $\rho(G)=0$ for a finite group $G$, then $G$ is a cut group, and  if $\rho(G)\leq 1$, we say that $G$ is an \ecut group.

		The class of \cut groups is extensively studied.  In the study, various equivalent criteria for a group $G$ to be a \cut group have been developed. An important role is played by a group theoretic criterion for $G$ to be a \cut group, which is as follows:

		\begin{theorem}[\cite{DMS05}, Theorem 1 and Lemma 2, see also \cite{MP18}, Theorem 5] The following are equivalent for a finite group $G$:
			
			\begin{enumerate}
				\item[(i)] $G$ is a \cut group.
				\item[(ii)] for every $x\in G$ and for every  $j\in \mathbb{N}$ with $gcd(j, |G|) = 1$, either $x^j$$\sim$ $x$ or $x^j$$\sim $$x^{-1}$.
				\item[(iii)] for every $x\in G$ and for every  $j\in \mathbb{N}$ with $gcd(j, |x|) = 1$, either $x^j$$\sim$ $x$ or $x^j$$\sim $$x^{-1}$, i.e., $x$ is an \textit{inverse semi-rational element} of $G$. 
			\end{enumerate}

\end{theorem}
Hence, to say that $G$ is a \cut group is same as saying that $G$ is an inverse semi-rational group.

		We provide analogous criterion for \ecut groups.
		
		\subsection{Conjugacy criterion}
		 Ferraz \cite{Fer04} proved that for a finite group $G$, the number of simple components of the Wedderburn decomposition of $\mathbb{R}G$ (respectively $\mathbb{Q}G$) equals the number of $\R$-classes (respectively $\Q$-classes) in $G$ and \begin{equation}\label{rankbydiffofcc}
			\rho(G)= n_\mathbb{R} - n_\mathbb{Q},
			\end{equation}  
		where  $n_\mathbb{R}$ (respectively $n_\mathbb{Q}$) denotes the number of $\mathbb{R}$-classes (respectively $\mathbb{Q}$-classes) in $G$. Clearly, for any $g\in G$, $C_{g} \subseteq C_{g}^{\mathbb{R}} \subseteq C_{g}^{\mathbb{Q}}$ and hence $n_\mathbb{R} \geq n_\mathbb{Q}$.

		If $G$ is a \cut group, then $n_\mathbb{R} =n_\mathbb{Q}$ and converse holds. 
		Similarly, $G$ is an \ecut group if and only if $n_\mathbb{R} - n_\mathbb{Q} \leq 1$. Let $S$ be the set of all inverse semi-rational elements in $G$. Clearly, $G$ is a \cut group if and only if $G=S$. Otherwise, for any $x\in G\setminus S$, there exists $j(\neq\pm 1)$, with $\gcd(j,|x|)=1$,  such that $x^j\nsim x$ and $x^j \nsim x^{-1}$, i.e., $x^j$ and $x$ have disjoint $\mathbb{R}$-classes for such $j$. If $G$ is an \ecut group which is not \cut, then $n_\mathbb{R} - n_\mathbb{Q} = 1$. Hence, we must have $C_x^{\mathbb{Q}}= C_{x^j}^{\mathbb{R}} ~\dot{\cup}~ C_{x}^{\mathbb{R}}$ for any such $j$, where $\dot{\cup}$ denotes the disjoint union. Clearly, if $x\in G\setminus S$, then $C_x^{\mathbb{Q}}\subseteq G\setminus S$. Rather, $G\setminus S=C_x^{\mathbb{Q}}$. Because, for any $y\in G\setminus S$, the $\mathbb{Q}$-class of $y$ splits into at least two $\mathbb{R}$-classes, which contradicts $n_\mathbb{R} - n_\mathbb{Q} \leq 1$, if $y\not \in C_x^{\mathbb{Q}}$. Hence, we conclude the following theorem.
		 
		 \begin{theorem}\label{theorem_CC}
		 	
		 	Let $G$ be a finite group and let $S$ be the set of all inverse semi-rational elements in $G$. Then, $G$ is an \ecut group if and only if one of the following holds:
		 	\begin{enumerate}
		 		\item[(i)] $G=S$. In this case, $G$ is a \cut group.
		 		\item[(ii)] $G\neq S$ and $G\setminus S=C_x^{\mathbb{Q}}=C_x^{\mathbb{R}}~\dot{\cup}~ C_{x^j}^\mathbb{R}$, 
		 		where $x$ is any element of $G\setminus S$ and \linebreak $j\in \{ 1 < i < |x|:gcd(i,|x|)=1, ~x^i\nsim x, ~x^i\nsim x^{-1}\}$.
		 	\end{enumerate}

		 \end{theorem}
Note that if $G$ is not a \cut group, then clearly, there exist elements in $G$ whose $\mathbb{Q}$-classes are not same as their $\mathbb{R}$-classes. \Cref{theorem_CC} states that if $G$ is an \ecut group which is not a \cut group, then all such elements, i.e., elements which are not inverse semi-rational in $G$, form a single $\mathbb{Q}$-class which is a disjoint union of precisely two $\mathbb{R}$-classes.\\

	 The following corollary follows directly from the above theorem. 
	 
	 \begin{corollary}\label{cyclic_ecut}
	 	 A cyclic group $C_n$ of order $n$ is an \ecut group if and only if \linebreak $n \in \{1,2,3,4,5,6,8,12\}.$
	 \end{corollary}
		 
		  If $G$ is an abelian group, then it is a direct product of cyclic groups. We recall that being \cut group (consequently, being \ecut group) is not a direct product closed property (\cite{BMP17}, Remark 1), unless the component groups satisfy certain requisite conditions (\cite{Mah18}, Section 3). It is of natural interest to study the behaviour of \ecut groups with respect to direct products, under suitable restrictions. \\
		  
		  To proceed further, we first recall another equivalent criterion for a group $G$ to be a \cut group.
		  
		  \begin{theorem}\label{theorem_SC_cut}(\cite{MP18}, Theorem 5)
		  	Let $G$ be a finite group.  Then the following are equivalent:
		  	
		  	\begin{enumerate}
		  		\item[(i)] $G$ is a \cut group.
		  		\item[(ii)] The character field $\mathbb{Q}(\chi):=\{\mathbb{Q}({\chi(g)~|~ g\in G})\}$ of each absolutely irreducible character $\chi$ of $G$ is either $\mathbb{Q}$ or an imaginary quadratic field.
		  		\item[(iii)] If $\mathbb{Q}G\cong \bigoplus_i M_{n_i}(D_i)$ is the Wedderburn decomposition of $\mathbb{Q}G$, where the simple component $M_{n_i}(D_i)$ 
		  		denotes the algebra of  $n_i\times n_i$ matrices over the division ring $D_i$, then the centre $\mathcal{Z}(D_i)$ of each division ring $D_i$ is $\mathbb{Q}$ or an imaginary quadratic field.
		  	\end{enumerate}

		  \end{theorem}

		{\subsection{Simple components}
			
		If $G$ is an \ecut group, we study the simple components of $\mathbb{Q}G$ analogous to Theorem \ref{theorem_SC_cut}.
			\begin{theorem}\label{theorem_SC_ecut}
				A finite group $G$ is an e-cut group if and only if the centre of every simple component of $\mathbb{Q}G$ is either $\mathbb{Q}$ or $\mathbb{Q}(\sqrt{-d})$ for some $d >0$, except possibly for one simple component for which the centre may be either real quadratic extension of $\mathbb{Q}$ or a degree $4$ extension of $\mathbb{Q}$ with no real embeddings. 
			\end{theorem}
			\begin{proof} For a group $G$, the centre $\mathcal{Z}(\mathbb{Z}G)$ is an order in $\mathcal{Z}(\mathbb{Q}G)$ which equals $\oplus_{\chi\in \Irr(G)} \Q(\chi)$. Thus, $\mathcal{Z}(\mathbb{Z}G)$ is of finite additive index in the unique 
				maximal order $\oplus_{\chi\in \Irr(G)} \mathcal{O}_\chi$ of $\oplus_{\chi\in \Irr(G)} \Q(\chi)$, where $\mathcal{O}_\chi$ denotes the ring of integers in $\Q(\chi)$, for  ${\chi\in \Irr(G)}$. Therefore, the unit group of $\mathcal{Z}(\ZZ G)$ is of finite index in the multiplicative group
				$\oplus_{\chi\in \Irr(G)} {\U}(\mathcal{O}_\chi)$ and hence $\mathcal{Z}({\U}(\ZZ G))$ has the same rank as $\oplus_{\chi\in \Irr(G)} {\U}(\mathcal{O}_\chi)$. Then, 
				the condition that $G$ is \ecut translates to $\U(\mathcal{O}_\chi)$ being finite for every $\chi\in \Irr(G)$, except possibly for one $\chi$ for which
				$\rho(\U(\mathcal{O}_\chi) )=1$.
				
				 For $\chi \in \Irr(G) $, if $\rho(\U(\mathcal{O}_\chi) )=0$, then it follows from Dirichlet's unit theorem, that $\mathbb{Q}(\chi)$ is either $\mathbb{Q}$ or  $\mathbb{Q}(\sqrt{-d})$ for some $d >0$. Also, for a $\chi' \in \Irr(G)$, if  $\rho(\U(\mathcal{O}_{\chi'}) )=1$, then Dirichlet's unit theorem yields that $\mathbb{Q}(\chi')$ is either real quadratic extension of $\mathbb{Q}$ or a degree $4$ extension of $\mathbb{Q}$ with no real embeddings. Clearly, the converse holds.
				
			\end{proof} 
		

	\subsection{Direct products}Let $G=H\times K$ be the direct product of two non-trivial groups $H$ and $K$. If $G$ is \cut, then by Theorem \ref{theorem_SC_cut}, it follows that $H$ and $K$ are \cut groups. The converse is obviously not true. For instance, $C_3\times C_4$ is not a \cut group. However, if one of $H$ or $K$ is real (and hence rational), then it follows by the definitions that $H\times K$ is a \cut group.
		
		Likewise, if $G$ is \ecut, then by Theorem \ref{theorem_SC_ecut}, it follows that $H$ and $K$ are \ecut groups. The converse is not true even if one of $H$ or $K$ is rational. For instance, $C_5\times C_2$ is not an \ecut group. Rather, we observe that if $H\times K$ is an \ecut group, then both $H$ and $K$ have to be \cut groups. Indeed, let $H$ and $K$ be \ecut groups such that $H$ is an \ecut group which is not \cut, so that there exists an irreducible character $\theta$ of $H$ such that $\mathbb{Q}(\theta)$ is neither $\mathbb{Q}$ nor an imaginary quadratic field. Since, $|K|\geq 2$, let $\psi_1$ and $\psi_2$ be two distinct irreducible characters of $K.$ Now, the characters $\theta \times \psi_1$ and $\theta \times \psi_2$ $\in  \Irr(H\times K)$ are such that they yield two character fields different from $\mathbb{Q}$ or imaginary quadratic, contradicting the fact that $H\times K$ is \ecut. It follows that if an \ecut group $G=H_1\times H_2 \times... \times H_r$, ($r>1$) is a direct product of non-trivial groups, then each $H_i$ is a \cut group. In other words, we have the following:

		\begin{remark}\label{Remark_DirectProduct}
		If $G$ is an \ecut group, which a direct product of groups $H_1$, $H_2$,....,$H_r$, then either each $H_i$, $1\leq i \leq r$, is a \cut group, or if  for some $j$,  $H_j$ is an \ecut group but not a \cut group,  then $G=H_j$, i.e., each $H_i$, $i\neq j$ is trivial.
		\end{remark} 
	
	We use this remark to classify abelian \ecut groups.\\
		
		Let $G$ be an abelian group, so that we may write $G$ as direct product of cyclic groups, $G\cong C_{n_1} \times C_{n_2}\times...\times C_{n_k}$, where $n_i$ divides $n_{i-1}$ for each $i$. Now, if there exists $j$, such that $C_{n_j}$ is not \cut, then $G=C_{n_j}$, with $n_j \in \{5,8,12\}$.  Otherwise, each $C_{n_i}$ is a \cut group, and hence, exponent of $G$, which equals $n_1$ is $2,3,4$ or $6$. Conversely, any abelian group of exponent $2,3,4$ or $6$ is \cut group and consequently \ecut and cyclic groups of order $5,8$ and $12$ are also \ecut. Thus, we conclude the following:
		
		\begin{theorem}\label{Theorem_abelian_ecut}
An abelian group $G$ is an \ecut group if and only if either $G$ is a cyclic group of order $5,8$ or $12$,  or the exponent of $G$ divides $4$ or $6$.
		\end{theorem}

It may be noted that if $G$ is an abelian group, then the expression of $\rho(G)$ is known in (\cite{AA69}, Theorem 4) and one may classify abelian \ecut groups using it.\\

	In the next section, we study nilpotent \ecut groups.

		\section{Nilpotent e-cut groups}
		In this section, we give classification of finite nilpotent groups $G$ which are \ecut. We begin by discussing finite \ecut $p$-groups.

	\subsection{ \ecut $p$-groups} \begin{theorem}\label{theorem_classification_pgroups} If $G$ is an \ecut $p$-group, then one of the following holds:
		
	 \begin{itemize} \item [(i)] $G \cong C_{5}$; \item [(ii)] $G$ is a \cut $3$-group; \item [(iii)] $G$ is a \cut 2-group;
			\item[(iv)] $G$ is a $2$-group and has an element $x$ of order $8$ such that $C_x^\mathbb{R}~\dot{\cup}~C_{x^3}^\mathbb{R}$ contains all the elements of $G$ which are not inverse semi-rational.

		
	 \end{itemize}\end{theorem}\begin{proof} 	Clearly, a group  $G$ satisfying any of the conditions (i)-(iv) is \ecut. Now, let $G$ be a $p$-group which is \ecut. As $\mathcal{Z}(G)$ is an abelian \ecut $p$-group, it follows that  $p\in \{2,3,5\}$ (by Theorem \ref*{Theorem_abelian_ecut}). We thus have the following three cases: \\
  \textbf{Case-I} $p=5$ \\ Here, $G/G'$ is an abelian $5$-group, and hence $G/G' \cong C_{5}$ (\Cref{Theorem_abelian_ecut}). One can verify that if $G$ is a finite nilpotent group such that $G/G'$ is cyclic then, $G$ is abelian. Consequently, $G \cong C_{5}$. \\ \textbf{Case-II} $p=3$ \\ Suppose $G$ is a $3$-group which is \ecut but not \cut. Then there exists $x \in G $ which is not inverse semi-rational. It follows from the proof of (\cite{BMP17}, Theorem 2) that $x^2\nsim x^{-1}$, i.e, $x$ and $x^2$ have disjoint $\mathbb{R}$-classes. But, then Theorem \ref{theorem_CC}(ii) is not satisfied, as in this case $x^4\not\in C_x^\mathbb{R}~\dot{\cup}~C_{x^2}^\mathbb{R}$.\\ \textbf{Case-III} $p=2$\\ Let $G$ be a $2$-group which is \ecut but not \cut. Then there exists $x \in G $ such that $x^{3}$ is not conjugate to $x$ or $x^{-1}$. Hence, by Theorem \ref*{theorem_CC}(iii), $C_x^\mathbb{Q}=C_x^\mathbb{R}~\dot{\cup}~C_{x^3}^\mathbb{R}$ is the set of all elements in $G$ which are not inverse semi-rational.

			\end{proof}
		
			  \subsection{Nilpotent \ecut groups} \begin{theorem} A finite nilpotent group $G$ is an e-cut group if and only if one of the following holds: \begin{itemize}\item [(i)] $G \cong C_{5};$ \item [(ii)] $G$ is a \cut $3$-group; \item [(iii)] $G$ is an \ecut $2$-group;  
		  		\item [(iv)] $G=H \times K$, where $H$ is a \cut $2$-group and $K$ is a \cut $3$-group such that
		  		
		  		\begin{enumerate}
		  			\item[(a)] either $H$ is real, i.e., $h\sim h^{-1}$ for every $h\in H$, or\item[(b)] there exists $h\in H$ such that $h\nsim h^{-1}$ and and for any $k\in K$, \linebreak $C_{(h,k)}^\mathbb{Q}=C_{(h,k)}^\mathbb{R}~\dot{\cup}~C_{{(h,k^{-1})}}^\mathbb{R}$ forms the set of all elements in $G$ which are not inverse semi-rational. 
		  		\end{enumerate}

		  	 \end{itemize} \end{theorem} 
	  	 
	  	 \begin{proof} Clearly, if $G$ satisfies one of (i)-(iv), then $G$ is an \ecut group. Conversely, suppose that $G$ is a nilpotent group which is \ecut. For a prime $p$ dividing $|G|$, let $P_p$ denote the Sylow $p$-subgroup of $G$. It follows from \Cref{Remark_DirectProduct} that $G=P_2\times P_3 \times P_5$. Further, if $P_5$ is non-trivial, then by \Cref{Remark_DirectProduct} and \Cref{theorem_classification_pgroups}(i), $G\cong C_5$, else $G=P_2\times P_3$. Again, by \Cref{Remark_DirectProduct},  if $P_3$ is \ecut, then $G=P_3$. Consequently, by \Cref{theorem_classification_pgroups}, we have that $G$ is a \cut $3$-group. Again, if $P_2$ is an \ecut group, then $G=P_2$, i.e., $G$ is an \ecut $2$-group. Otherwise, $G=P_2\times P_3$, where each $P_2$ and $P_3$ is a \cut group. Further, if $P_2$ is a real (or rational) group, it follows from (\cite{Mah18}, Theorem 3) that $P_2\times P_3$ is a \cut and hence an \ecut group. Therefore, we are now left with the case when $G=P_2\times P_3$, with $P_3$ a \cut group and $P_2$ a \cut group which is not rational. For notational convenience, denote $P_2$ by $H$ and $P_3$ by $K$. Hence, in this case, there exists $h\in H$ such that $h\nsim h^{-1}$. Direct computations yield that for any $k\in K$,  $C_{(h,k)}^\mathbb{Q}=C_{(h,k)}^\mathbb{R}~\dot{\cup}~C_{{(h,k^{-1})}}^\mathbb{R}$ is the set of all elements in $G$ which are not inverse semi-rational.

	  	 
  	  \end{proof} 
    

		\section{Monomial \ecut groups}
		
		In this section, we study monomial \ecut groups. In  particular, we classify split metacyclic \ecut groups. To begin with, we consider dihedral and generalized quaternion groups. \\

			\subsection{Dihedral and generalized quaternion \ecut groups}
		
		\begin{proposition}\label{dihedral_quaternion_ecut}
			
			Let $D_{2n}:=\langle a,b : a^n=1,b^2=1, a^b=a^{-1}\rangle $ be a dihedral group of order $2n$, where $n\geq 3$ and let $Q_{4m}:=\langle a,b ~|~a^{2m}=1, b^{2}=a^m, a^b=a^{-1}\rangle$ be a generalized quaternion group of order $4m$, where $m\geq 2$. Then,
			
\begin{enumerate}
	\item[(i)] $D_{2n}$ is an \ecut group if and only if $n\in \{3,4,5,6,8,12\}$.
\item[(ii)] $Q_{4m}$ is an \ecut group if and only if $m\in \{2,3,4,6\}$.
\end{enumerate}

		\end{proposition}
		 \begin{proof}The dihedral group $D_{2n}$ is \ecut if and only if $n_\mathbb{R}-n_\mathbb{Q}\leq 1$, where $n_\mathbb{R}$ is the number of distinct $\mathbb{R}$-classes in $G$ and $n_\mathbb{Q}$ is the number of distinct $\mathbb{Q}$-classes in $G$. Note that the elements of the form $a^ib\in D_{2n}$, $1\leq i \leq n$ have order $2$  and hence their $\mathbb{R}$-class coincides with $\mathbb{Q}$-class. We thus only consider the elements of order $a^\alpha$, where $1< \alpha <n$. The $\mathbb{R}$-class of $a^\alpha$ is $\{a^{\alpha}, a^{-\alpha}\}$, which is also its conjugacy class and the $\mathbb{Q}$-class of $a^\alpha$ equals $\{a^{\alpha i}~|~\gcd(i,|a^\alpha|)=1\}$. Therefore, $\mathbb{Q}$-class and the $\mathbb{R}$-class coincide if and only if $\phi(|a^\alpha|)=2$, where $\phi$ denotes the Euler totient function. In other words, $D_{2n}$ is \cut if and only if $n\in \{2,3,4,6\}$. \\

		 	 
		 	  Further, if $D_{2n}$ is an \ecut group which is not \cut, then  $\phi(|a^\alpha|)=4$ for precisely one choice of $\alpha$, so that for that $\alpha$, the $\mathbb{Q}$-class of $a^\alpha$ has exactly $4$ elements and it can be written as disjoint union of precisely two $\mathbb{R}$-classes. Hence, $|a^\alpha|\in \{5,8,10,12\}$ for exactly one choice of $\alpha$. However, if $|a^\alpha|=10$, then $|a^{2\alpha}|=5$ which implies that there are two such elements. Thus $|a^\alpha|\neq 10$ and hence $n\in \{5,8,12\}$. Conversely, it can be directly verified that if $n\in \{5,8,12\}$, then $D_{2n}$ is an \ecut group. Hence, (i) is proved. The proof for $Q_{4m}$ is analogous and we omit the details. 
		 	
 \end{proof}

	\subsection{Strongly monomial \ecut groups}
		Jespers et al. \cite{JOdRG13} gave a formula to compute the  rank $\rho(G)$ of $\mathcal{Z}(U(\mathbb{Z}G))$ for a large class of groups which include abelian-by-supersolvable groups. In order to present the same, we need to recall strongly monomial groups.
		
		\begin{definition}[Strongly monomial groups]
	
		Let $G$ be a finite group. 	For $K\unlhd H\leq G$, set  $$\epsilon(H,K)=\begin{cases}
			\hat{H}, ~\mathrm{if}~ H=K,\\
			\Pi (\hat{K}-\hat{L}), ~\mathrm{otherwise},
		\end{cases}$$ where $\hat{H}=\frac{1}{|H|}\sum_{h\in H}h$ and $L$ runs 
over the minimal normal subgroups of $H$, which contain $K$ properly. A strong Shoda pair \cite{OdRS04} of $G$ is a pair $(H, K)$ of subgroups of $G$ satisfying:
		\begin{enumerate}
			\item[(i)] $K\unlhd H\unlhd N_G(K)$, where $N_G(K)$ is the normalizer of $K$ in $G$;
			\item[(ii)] $H/K$ is cyclic and a maximal abelian subgroup of $N_G(K)/K$; and
			\item[(iii)] the distinct $G$-conjugates of $\epsilon(H,K)$ are mutually orthogonal.
		\end{enumerate}
	
	If $H$ is a normal subgroup of $G$, then (iii) follows, and $(H,K)$ is called as an extremely stong Shoda pair of $G$. 
			
		It is known that if $(H, K)$ is a strong Shoda pair of $G$, then $e(G, H, K)$, the sum of all distinct $G$-conjugates of $\epsilon(H,K)$, is a primitive central idempotent of the rational group algebra $\mathbb{Q}G$ (\cite{JdR16}, Proposition 3.3). Two strong Shoda pairs $(H_1, K_1)$ and $(H_2, K_2)$ are said to be equivalent, if $e(G, H_1, K_1) = e(G, H_2, K_2)$. A complete set of inequivalent strong Shoda pairs of $G$ is denoted by $\mathcal{S}(G)$. A finite group $G$ is said to be strongly monomial, if every primitive central idempotent of $\mathbb{Q}G$ is of the form $e(G, H, K)$ for some strong Shoda pair $(H, K)\in \mathcal{S}(G)$. 	\end{definition}

		A rank formula for strongly monomial groups is stated below:
		
		\begin{theorem}[\cite{JOdRG13}, Theorem 3.1]\label{rank_strongly_monomial}
			
			Let $G$ be a finite strongly monomial group. Then $$\rho(G)=\sum_{(H,K)\in\mathcal{S}(G)} \left(\frac{\phi([H:K])}{k_{(H,K)}[N_G(K):H]}-1\right),$$
			
			where $H/K =\langle hK\rangle $ and
			$k_{(H,K)}=\begin{cases}
				1, ~\mathrm{if} ~hh^x\in K~\mathrm{for~ some}~x\in N_{G}(K),\\
				2, ~\mathrm{otherwise}.
			\end{cases} $
			
		\end{theorem}
Note that a generalization of strongly monomial groups, termed as generalized strongly monomial groups, is given in \cite{BK19-2}. This class of groups has been well studied in a series of papers \cite{BK19-2,BK22-2,Kau23} and is a substantial class of monomial groups. In \cite{BK22}, a precise formula to compute the rank of $\mathcal{Z}(\mathcal{U}(\mathbb{Z}G))$ when $G$ is a generalized strongly monomial group has been provided, which is a generalization of  \Cref{rank_strongly_monomial}. We frequently make use of the fact that if $|G|$ is odd, then $k_{(h,k)}=2$, as observed in \cite{BK22}.
			\par A group $G$ is said to be normally monomial, if every complex irreducible character of $G$ is induced from a linear character of a normal subgroup of $G$ \cite{How84c}. Hence, for a normally monomial group $G$, the set $\mathcal{S}(G)$ consists of the extremely strong Shoda pairs of $G$. Clearly, normally monomial groups are strongly monomial. An effective algorithm to compute $\mathcal{S}(G)$ for a given normally monomial group $G$ is given in \cite{BM14} (see also \cite{BM16}).
	
		Since all metabelian groups are normally monomial, we provide examples of certain families of metabelian \ecut groups, by computing $\rho(G)$ using above results.
		
		\begin{example}\begin{description}
				\item[(Cyclic groups)] Let $C_n=\langle a \rangle$ be a cyclic group of order $n$, where $n\geq 2$. Then, 
				
			\small{	$$\mathcal{S}(C_n)={\underset{d|n}{\bigcup}}\{(\langle a \rangle, \langle a^d\rangle)\}~\mathrm{and }~ \rho(C_{n})=\lfloor\frac{n}{2} \rfloor +1,$$} where $\lfloor x \rfloor$ denotes the least integer greaater than or equal to $x$.\\
			
				\item[(Dihedral groups)] 
				
				Let $D_{2n}=\langle a,b : a^n=1,b^2=1, a^b=a^{-1}\rangle $ be a dihedral group of order $2n$, where $n\geq 3$. Then,
				
			\small{$$\mathcal{S}(D_{2n})=\begin{cases}
					\{(D_{2n},D_{2n}),(D_{2n},\langle a\rangle)\} {\underset{\substack{d|n\\d\neq 1}}{\bigcup}}\{(\langle a \rangle,\langle a^d\rangle)\},~ \mathrm{if}~ 2\nmid n;\\
					\{(D_{2n},D_{2n}),(D_{2n},\langle a^2,ab\rangle),(D_{2n},\langle a^2,b\rangle),(D_{2n},\langle a\rangle)\}{\underset{\substack{d|n\\d\neq 1,2}}{\bigcup}}\{(\langle a \rangle,\langle a^d\rangle)\}, \mathrm{if} ~2\mid n.
				\end{cases}$$}

				\noindent We observe that the only strong Shoda pairs that contribute to $\rho(D_{2n})$ are $(\langle a \rangle,\langle a^d\rangle), $ where $d|n, d\neq 1,2$ and hence, $\rho(D_{2n})=\lfloor\frac{n}{2} \rfloor +1$.\\
				\item[(Quaternion groups)] Likewise, if we consider generalized quaternion group $Q_{4m}$ of order $4m$, $m\geq 2$ presented by 
				$\langle a,b ~|~a^{2m}=1, b^{2}=a^m, a^b=a^{-1}\rangle,$   then
				
				\small{$$\mathcal{S}(Q_{4m})=\begin{cases}
					\{(Q_{4m},Q_{4m}),(Q_{4m},\langle a\rangle),(Q_{4m},\langle a^2\rangle)\}	{\underset{\substack{d|2m\\d\neq 1,2}}{\bigcup}}\{(\langle a \rangle,\langle a^d\rangle)\},\, \mathrm{if}~ 2\nmid m;\\
					\{(Q_{4m},Q_{4m}),(Q_{4m},\langle a^2,ab\rangle),(Q_{4m},\langle a^2,b\rangle),(Q_{4m},\langle a\rangle),(Q_{4m},\langle a^2\rangle)\}	{\underset{\substack{d|2m\\d\neq 1,2}}{\bigcup}}\{(\langle a \rangle,\langle a^d\rangle)\},\, \mathrm{if} ~2\mid m.
				\end{cases}$$}

				As in the case of dihedral groups, the only strong Shoda pairs that contribute to $\rho(Q_{4m})$ are $(\langle a \rangle,\langle a^d\rangle),~\mathrm{where}~ d|2m, d\neq 1,2$ which yields that $\rho(Q_{4m})=m+1-\tau(2m),$ where $\tau(x)$ denotes the number of positive divisors of $x$.
				
			\end{description}			
		\end{example}
	
By above example, it follows that for $n\in \mathbb{N}$, if $G$ is either $C_n$ or $D_{2n}$ or $Q_{2n}$, with $2|n$, then$$\rho(G)=\lfloor\frac{n}{2} \rfloor +1-\tau(n).$$ Consequently, $\rho(G)\leq 1$, if and only if  $$\tau(n)\geq \lfloor\frac{n}{2} \rfloor$$ which happens precisely when $n\in \{1,2,3,4,5,6,8,12\}$, as also obtained in \Cref{cyclic_ecut} and \Cref{dihedral_quaternion_ecut}.
	
	\begin{example} Let $p$ be a prime and let $G$ be a group of order $p^n$, where $n\in \mathbb{N}$. If $G$ is abelian, then \Cref{Theorem_abelian_ecut} is applicable to determine if $G$ is \ecut or not. The rank of non-abelian groups of order $p^n$, where $n=3$ or $4$, have been computed using \Cref{rank_strongly_monomial} in (\cite{BM15}, Section 4). We observe that, amongst these, the groups $G$ for which $\rho(G)\leq 1$ are  precisely the following:
		
		\begin{itemize}
			\item[(a)] all non-abelian groups of order $2^3,~3^3$ and $2^4$.
			\item[(b)]  three groups of order $3^4$, with presentations as listed below:
			\end{itemize}
	
		\small{
		\begin{itemize}
			   \item $\langle a, b, c: a^{9} =b^{3} =c^{3}=1, ca=a^{4}c, ba=ab, cb=bc \rangle;$
			
			\item $\langle a, b, c: a^{9} =b^{3} =c^{3}=1, ba=a^{4}b, ca=abc, cb=bc \rangle;$ and
		
			\item $\langle a, b, c, d: a^{3} =b^{3} =c^{3}=d^{3}=1, dc=acd, bd=db, ad=da, bc=cb,ac=ca,ab=ba \rangle;$
		\end{itemize}}

Amongst these, the groups for which $\rho(G)= 1$ (i.e., \ecut but not \cut) are only two groups of order $16$, with presentations given by

\begin{itemize}
	\item $\langle a, b : a^{8} =b^{2} =1, ba=a^{7}b  \rangle$; and 
	\item $\langle a, b : a^{8} =b^{4} =1, ba=a^{7}b, a^{4}=b^{2}  \rangle,$
\end{itemize}
which are respectively, the dihedral and the quaternion groups of order $16$ each.\\

Clearly, these results are in accordance with the ones obtained in \Cref{theorem_classification_pgroups}.\\

It may be noted that any group of order $p^5$, where $p$ is a prime, is metabelian and hence normally monomial. Therefore, these methods are applicable to check whether a given group of order $p^5$ is \ecut or not.
	\end{example}

	\begin{example}
	Let $G=C_{(p^n,q^m)}$ be a  metacyclic group of the form $ C_{q^m}\rtimes C_{p^n}$, where $p$ and $q$ are distinct  primes, $n,m\in\mathbb{N}$ and the cyclic group $C_{p^n}$ 
	acts faithfully on the cyclic group $C_{q^m}$, so that the presentation of $G$ is given by $$G:=\langle a,b ~|~a^{q^m}=b^{p^n}=1, a^b=a^r\rangle,$$ where $m,l,r\in \mathbb{N}$ are such that the multiplicative order of $r$ modulo $q^m$ equals $p^n$.\\
	
	The set $\mathcal{S}(G)$ of strong Shoda pairs of $G$, as computed in \cite{JOdRG13} is given by  
	
	$$\mathcal{S}(G):=\{(G,G)\}\bigcup_{j_1=1}^{m} \{(\langle a\rangle, \langle a^{q^{j_1}}\rangle)\}\bigcup_{j_2=1}^{n}\{(G,\langle a,b^{p^{j_2}} \rangle)\}.$$
	
Hence,	\begin{equation}\label{Ranks}
		\rho(C_{(p^n,q^m)}) = \begin{cases} 
			2^{n-1} + \frac{q^m-1}{2^n} -n-m, ~\mathrm{if} ~p=2, \\
			\frac{p^n-1}{2} + \frac{q^m-1}{2p^n} -n-m, ~\text{otherwise.} 
		\end{cases}
	\end{equation}

	\end{example}

Consequently, we have the following:

\begin{enumerate}
	\item[(i)] $C_{(p^n,q^m)}$ is \cut if and only if \small{$(p^n, q^m)\in \lbrace (2, 3), (3,7)\rbrace.$}
	\item[(ii)]  $C_{(p^n,q^m)}$ is \ecut if and only if \small{$(p^n, q^m)\in \lbrace (2, 3), (2,5), (4, 5), (3,7), (3,13), (5,11)\rbrace.$}\\
\end{enumerate}

 We now proceed to classify metacyclic \ecut groups.
		\subsection{Metacyclic \ecut groups}

	In this subsection, we give a complete classification of split metacyclic \ecut groups.
		Though computations of rank for this class of groups is observed in \cite{FS08}, it cannot be effectively used to list all such \ecut groups. 
		
		
		

		
		
		\begin{theorem}
			Let $G$ be a split metacyclic group. Then, $G$ is an \ecut group if an only if it is isomorphic to  $$G_{n,t,l,r}:=\langle a,b: a^n=1,b^t=a^l, a^b=a^r\rangle,$$ for the choices of parameters $n,t,l,r$ as listed in tables \ref{tab:ecut} and \ref{tab:cut}.
		\end{theorem}

	\begin{table}[H]
		\centering
		\begin{tabular}{|c|c|c|c|c|}\hline
			\textbf{n} & \textbf{l} & \textbf{r} & \textbf{t}& \textbf{GAP ID} \\\hline
			3	& 3	& 2	& 2	& [6,1]	\\\hline
			4	& 4	& 3	& 2	& [8,3]	\\\hline
			4	& 2	& 3	& 2	& [8,4]	\\\hline
			3	& 3	& 2	& 4	& [12,1]	\\\hline
			6	& 6	& 5	& 2	& [12,4]	\\\hline
			4	& 4	& 3	& 4	& [16,4]	\\\hline
			4	& 2	& 3	& 4	& [16,6]	\\\hline
			8	& 4	& 3	& 2	& [16,8]	\\\hline
			3	& 3	& 2	& 6	& [18,3]	\\\hline
			5	& 5	& 2	& 4	& [20,3]	\\\hline
			7	& 7	& 2	& 3	& [21,1]	\\\hline
			12	& 6	& 5	& 2	& [24,5]	\\\hline
			6	& 6	& 5	& 4	& [24,7]	\\\hline
			4	& 4	& 3	& 6	& [24,10]	\\\hline
			4	& 2	& 3	& 6	& [24,11]	\\\hline
			9	& 3	& 4	& 3	& [27,4]	\\\hline
			8	& 4	& 5	& 4	& [32,4]	\\\hline
			8	& 8	& 3	& 4	& [32,13]	\\\hline
			6	& 6	& 5	& 6	& [36,12]	\\\hline
			10	& 10	& 3	& 4	& [40,12]	\\\hline
			7	& 7	& 3	& 6	& [42,1]	\\\hline
			7	& 7	& 2	& 6	& [42,2]	\\\hline
			12	& 12	& 5	& 4	& [48,11]	\\\hline
			9	& 9	& 2	& 6	& [54,6]	\\\hline
			9	& 3	& 4	& 6	& [54,11]	\\\hline
			15	& 15	& 2	& 4	& [60,7]	\\\hline
			16	& 4	& 5	& 4	& [64,28]	\\\hline
			16	& 8	& 3	& 4	& [64,46]	\\\hline
			12	& 12	& 7	& 6	& [72,37]	\\\hline
			12	& 6	& 7	& 6	& [72,38]	\\\hline
			20	& 10	& 3	& 4	& [80,29]	\\\hline
			20	& 20	& 13	& 4	& [80,30]	\\\hline
			20	& 20	& 3	& 4	& [80,31]	\\\hline
			14	& 14	& 3	& 6	& [84,7]	\\\hline
			14	& 14	& 9	& 6	& [84,9]	\\\hline
			18	& 18	& 5	& 6	& [108,26]	\\\hline
			18	& 6	& 7	& 6	& [108,31]	\\\hline
			30	& 30	& 17	& 4	& [120,41]	\\\hline
			21	& 21	& 10	& 6	& [126,7]	\\\hline
			21	& 21	& 2	& 6	& [126,8]	\\\hline
			28	& 28	& 11	& 6	& [168,20]	\\\hline
			28	& 14	& 11	& 6	& [168,21]	\\\hline
			36	& 12	& 7	& 6	& [216,78]	\\\hline
			36	& 6	& 7	& 6	& [216,81]	\\\hline
			42	& 42	& 19	& 6	& [252,28]	\\\hline
			42	& 42	& 11	& 6	& [252,29]	\\\hline
		\end{tabular}
		\caption{\textbf{Finite split metacyclic groups of rank $0$}.}
		\label{tab:cut}
	\end{table}

	\begin{table}[H]
	\centering
	\begin{tabular}{|ccccc|}\hline
		\textbf{n} & \textbf{l} & \textbf{r} & \textbf{t}& \textbf{GAP ID}\\\hline
		5  & 5 & 4 & 2 & [10,1]	 \\\hline
		8	& 8	& 7	& 2	& [16,7]	 \\\hline
		8	& 4	& 7	& 2	& [16,9]	 \\\hline
		3	& 3	& 2	& 8	& [24,1]	 \\\hline
		12	& 6	& 11	& 2	& [24,4]	 \\\hline
		12	& 12	& 11	& 2	& [24,6]	 \\\hline
		8	& 4	& 3	& 4	& [32,15]	 \\\hline
		3	& 3	& 2	& 12	& [36,6]	 \\\hline
		13	& 13	& 3	& 3	& [39,1]	 \\\hline
		5	& 5	& 2	& 8	& [40,3]	 \\\hline
		24	& 6	& 5	& 2	& [48,5]	 \\\hline
		12	& 6	& 11	& 4	& [48,10]	 \\\hline
		8	& 4	& 3	& 6	& [48,26]	 \\\hline
		11	& 11	& 3	& 5	& [55,1]	 \\\hline
		5	& 5	& 2	& 12	& [60,6]	 \\\hline
		21	& 21	& 4	& 3	& [63,3]	 \\\hline
		13	& 13	& 4	& 6	& [78,1]	 \\\hline
		7	& 7	& 3	& 12	& [84,1]	 \\\hline
		9	& 9	& 2	& 12	& [108,9]	 \\\hline
		15	& 15	& 2	& 8	& [120,7]	 \\\hline
		21	& 21	& 5	& 6	& [126,9]	 \\\hline
		13	& 13	& 2	& 12	& [156,7]	 \\\hline
		40	& 20	& 13	& 4	& [160,67]	 \\\hline
		40	& 40	& 3	& 4	& [160,68]	 \\\hline
		28	& 14	& 3	& 6	& [168,7]	 \\\hline
		28	& 28	& 3	& 6	& [168,9]	 \\\hline
		60	& 60	& 17	& 4	& [240,120]	 \\\hline
		35	& 35	& 3	& 12	& [420,15]	 \\\hline
		39	& 39	& 2	& 12	& [468,30]	 \\\hline
	\end{tabular}
	\caption{\textbf{Finite split metacyclic groups of rank 1.}}
	\label{tab:ecut}
\end{table}

	In Tables \ref{tab:cut} and  \ref{tab:ecut} , the first four columns list the identifying parameters $n,t,l, r$ of the group and the last column lists the identifier from library of small groups in GAP \cite{GAP4}.
		\begin{proof}
			Let $G$ be a split metacyclic group. Then $G$ has presentation $G_{n,t,l,r}$, for suitable choice of parameters $n,t,l,r$, i.e., $n\in \mathbb{N}$, $\gcd(r,n)=1$, $lr\equiv l (\mathrm{mod}~ n)$ and multiplicative order of $r$ modulo $n$ divides $t$. Without loss of generality, we may assume $l|n$. 
			
			If $G$ is an \ecut group, then so is $G/\langle a\rangle\cong C_t$ and thus $t\in \{2,3,4,5,6,8,12\}$. Denote the multiplicative order of $r$ modulo $n$ by $o$. We first verify that $(\langle a, b^o\rangle, \langle b^o\rangle)$ is an extremely strong Shoda pair of $G$, i.e.,
			
			\begin{itemize}
				\item[(i)] $\langle b^o\rangle\unlhd \langle a, b^o\rangle\unlhd G$.\\
				
				\item[(ii)] $\frac{\GEN{ a, b^o}}{\GEN{b^o}} \text{  is cyclic and  a maximal abelian subgroup of }  \frac{N_G(\langle  b^o\rangle)}{\GEN{b^o}}.$
			\end{itemize} 
			
			Observe that $a^{b^o} = a^{r^o} = a$, which implies $\GEN{b^o}\unlhd \GEN{a, b^o}$ and 
			$\GEN{a, b^o} \unlhd G$. For (ii), clearly $\frac{\GEN{ a, b^o}}{\GEN{b^o}} \text{  is cyclic.}$ For verifying maximality, consider a subgroup $\GEN{ a, b^d}$, with minimal $d$, such that $$\frac{\GEN{ a, b^o}}{\GEN{b^o}} < \frac{\GEN{ a, b^d}}{\GEN{b^o}} \le   \frac{N_G(\langle  b^o\rangle)}{\GEN{b^o}}=\frac{\GEN{a,b}}{\GEN{b^o}}$$

		Then, $a^{b^d} = a^{r^d} \neq 1$, as $d< o$. But then $\GEN{ a, b^d}/\GEN{b^o}$ is not abelian. Hence, (ii) follows.  By (\cite{OdRS04}, Proposition 3.4), the simple component corresponding to $(\langle a, b^o\rangle, \langle b^o\rangle)$ is $M_o(\mathbb{F})$, with $[\mathbb{F}:\mathbb{Q}]=\frac{\phi(n)}{o}$, where $\phi$ is the Euler totient function. By Dirichlet's unit theorem, if $\rho(G)\leq 1$, then $\frac{\phi(n)}{o}\in \{1,2,3,4\}$. Since $o|t$ and $t\in \{2,3,4,5,6,8,12\}$, we have that $\phi(n)\in \{2,3,4,5,6,7,8,10,12,16,24\}$. Rather, $\phi(n)\in \{2,3,4,5,6,7,8,10,12,16,24\},$ as $\phi(n)$ is even. Consequently, we have finite choices of $n$. Direct computations then yield the corresponding choices of $n$ as following:\\
		
			$\begin{array}{c l}
			\phi(n)&n \\ 
			2 & \{3,4,6\}  \\ 
			4 &\{ 5,8,10,12\} \\
			6 & \{7,14,9,18\} \\
			8 & \{16,24,15,30, 20\} \\
			10 & \{11,22\} \\
			12 & \{13,26,21,42,28,36\} \\
			16 & \{17,34,32,48,60,40\} \\
			18 & \{19,27,38,54\} \\
			24 &\{ 35,56,70,84,45,72,90,39,52,78\}\\
		\end{array}$
	
	\vspace{0.5cm}
	A systematic check done on above finite restricted parameters, for $G$ to be an \ecut group is done using GAP function \texttt{RankOfCentralUnits} (\cite{BBM20}, Section 3) to obtain Tables 1 and 2.
	\end{proof}

	\subsection{An \ecut group of order $1000$}
	\begin{example}	In \cite{OdRS04}, it has been pointed out that $G= \texttt{SmallGroup(1000,86)}$, as listed in GAP library of small groups, is the smallest monomial group which is not strongly monomial. In \cite{BK22}, it has been verified to be a generalized strongly monomial group and consequently, its rank $\rho(G)$ is computed to be $1$, i.e., $G$ is an \ecut group. 
	\end{example}

	\section{Solvable e-cut groups}
	
 Let $G$ be a finite solvable group and let $\pi(G)$ denote the set of primes dividing the order of $G$. Gow proved that if $G$ is rational, then $\pi(G)\subseteq\{2,3,5\}$  \cite{Gow}. In recent years, it has been proved that if $G$ is \cut, then $\pi(G)\subseteq\{2,3,5,7\}$ (\cite{Mah18}, \cite{Bac19}), which was deduced from the work of Chillag and Dolfi  \cite{CD10} who proved that if $G$ is a ﬁnite semi-rational solvable group, then $\pi(G)\subseteq\{2,3,5,7,13,17\}$.  We now try to analyse the prime spectrum of a solvable \ecut group.
 
 	Firstly, observe that for a group $G$, being semi-rational or quadratic rational, does not imply that $G$ is \ecut. For instance, the examples of groups given in \cite{Ten12}, namely
 $$G_1 = \langle a, b, c~|~a^2 = b^2 = c^8 = 1,~[b, c] = 1,~b^a = bc^4, c^a = c^3 \rangle ,$$ and 
 $$G_2 = \langle a, b, c~|~a^2 = b^2 = c^8 = 1, [b, c] = 1, b^a =b, c^a = bc^3 \rangle,$$ are both \ecut groups of order 32. However, $G_1$ is semi-rational but not quadratic rational and $G_2$ is quadratic rational but it is not semi-rational.

 Let $S$ be the set of all inverse semi-rational elements in $G$. If $G=S$, i.e., $G$ is \cut, then $\pi(G)\subseteq\{2,3,5,7\}.$ Otherwise, $G\setminus S=C_g^\mathbb{Q}$ for some $g\in G$. If $g$ is real, i.e., $g\sim g^{-1}$, then $g$ is semi-rational and consequently, $\pi(G)\subseteq\{2,3,5,7,13,17\}$.  
 Rather, if there exists a prime $p$, such that $ ~p\not\in \{2,3,5,7,13,17\}$ and $p||G|$, then there exist non real elements in $G$, and $G$ has exactly one irreducible character $\chi$ for which
 $\mathbb{Q}(\chi)$ is degree $4$ extension of $\mathbb{Q}$, with two pairs of conjugate embeddings. Hence, $\pi(G)\subseteq \{2,3,5,7,13,17,p_G\}$, where $p_G$ is a prime dividing $|G|$. It follows from \cite{F86} that $11\leq p_G \leq 267.$

	\section*{Acknowledgement(s)}

	The second author acknowledges the research support of the Department of Science and Technology (INSPIRE Faculty No. DST/INSPIRE/04/2023/001200), Govt. of India. The third author acknowledges the support by Science \& Engineering Research Board (SERB), Department of Science and Technology (DST), India (SRG/2023/000180).

\providecommand{\bysame}{\leavevmode\hbox to3em{\hrulefill}\thinspace}
\providecommand{\MR}{\relax\ifhmode\unskip\space\fi MR }
\providecommand{\MRhref}[2]{%
	\href{http://www.ams.org/mathscinet-getitem?mr=#1}{#2}
}
\providecommand{\href}[2]{#2}

\end{document}